\documentclass{article}
\usepackage{amssymb,amsmath,amsthm,graphicx}

\def\qed{\hfill {\hbox{${\vcenter{\vbox{               
   \hrule height 0.4pt\hbox{\vrule width 0.4pt height 6pt
   \kern5pt\vrule width 0.4pt}\hrule height 0.4pt}}}$}}}

\newtheorem{theorem}{Theorem}

\newtheorem{proposition}[theorem]{Proposition}
\newtheorem{corollary}[theorem]{Corollary}

\theoremstyle{definition}
\newtheorem{example}{Example}
\newtheorem{definition}{Definition}
\newtheorem{remark}{Remark}
\date{}

\title{\Large \textbf{Quantum Enhancements via Tribracket Brackets}}

\author{Laira Aggarwal\footnote{Email: laggarwal19@students.claremontmckenna.edu}\and
Sam Nelson\footnote{Email: Sam.Nelson@cmc.edu. Partially supported by Simons Foundation Collaboration Grant 316709}\and
Patricia Rivera\footnote{Email: patriciariv25@gmail.com}}

\begin{document}
\maketitle

\begin{abstract}
We enhance the tribracket counting invariant with \textit{tribracket brackets},
skein invariants of tribracket-colored oriented knots and links analogously
to biquandle brackets. This infinite family of invariants includes the 
classical quantum invariants and tribracket cocycle invariants as special cases,
as well as new invariants. We provide explicit examples as well as questions 
for future work.
\end{abstract}

\parbox{5.5in} {\textsc{Keywords:} 

}

\section{\large\textbf{Introduction}}\label{I}

In recent decades much work has involved colorings
of knot and link diagrams by algebraic structures such as \textit{quandles}
and \textit{biquandles}. These colorings have an interpretation as elements
of a homset collecting homomorphisms from an object associated to the knot
or link such as the knot group, the knot quandle etc. to the finite coloring
object. This interpretation gives rise to the notion of \textit{enhancements},
stronger invariants obtained from the set of colorings by collecting values
of an invariant of colored knots, starting with \textit{quandle cocycle 
invariants} of \cite{CJKLS} and explored in many subsequent works. See 
\cite{N1,EN} for an overview.

In \cite{NR}, the second listed author and a coauthor proposed the notion
of \textit{quantum enhancements}, quantum invariants of biquandle-colored knots
and links. In \cite{NOR}, the second listed author and coauthors introduced
a class of quantum enhancements known as \textit{biquandle brackets}, skein
invariants of biquandle-colored knots and links. This infinite class of
invariants includes the classical quantum invariants (Alexander-Conway, Jones,
HOMFLYPT and Kauffman polynomials) and the quandle cocycle invariants as special
cases, explicitly uniting these \textit{a priori} unrelated invariants;
moreover, examples of biquandle brackets which are apparently
neither of these types were identified
in \cite{NOR}. In \cite{NOSY} biquandle brackets were extended to define
\textit{biquandle virtual brackets} using a skein relation including a virtual
crossing, in \cite{IM} a type of biquandle bracket with skein relation 
involving graphs was introduced, and in \cite{NO} a method of
computing biquandle bracket invariants recursively with \textit{trace diagrams}
as opposed to via the original state-sum definition was introduced. See 
\cite{N2} for a survey of biquandle bracket invariants.

In \cite{NB}, colorings of knot diagrams by \textit{knot-theoretic ternary
quasigroups} or as we call them, \textit{Niebrzydowski tribrackets}, were
introduced. In this paper we extend the biquandle bracket idea to the
case of tribrackets, defining quantum enhancements via \textit{tribracket 
brackets}. The paper is organized as follows. In Section \ref{B} we review
the basics of tribracket theory. In Section \ref{TB} we introduce tribracket
brackets and provide some examples, including an illustration of the 
computation of the new invariant, examples showing that the new invariant is 
proper enhancement, and the values of the tribracket bracket invariant for 
two brackets on all prime knots with up to eight crossings and prime links 
with up to 7 crossings.
 In Section \ref{Q} we close with some
questions for future research.

\section{\large\textbf{Tribrackets}}\label{B}

We begin with a definition. See \cite{NB, NOO} etc. for more.

\begin{definition}
Let $X$ be a set. A \textit{horizontal tribracket} structure on $X$
is a ternary operation $[,,]:X\times X\times X\to X$ satisfying the 
conditions
\begin{itemize}
\item[(i)] In the equation $[a,b,c]=d$, any three of the variables
determine the fourth, and
\item[(ii)] For all $a,b,c,d\in X$ we have
\[[c,[a,b,c],[a,c,d]]=[b,[a,b,c],[a,b,d]]=[d,[a,b,d],[a,c,d]].\]
\end{itemize}
\end{definition}

The condition (i) makes $X$ a \textit{ternary quasigroup}, meaning that
if we think of $[a,b,c]$ as a three term product, there are well-defined
operations of left, center and right division. That is, each of the equations
\[[x,b,c]=d,\quad [a,x,c]=d \quad\mathrm{and}\quad [a,b,x]=d\]
has a unique solution in $X$. The condition (ii) makes our structure
a \textit{knot-theoretic quasigroup} in the terminology of \cite{NB},
which we call a \textit{Niebrzydowski tribracket} following \cite{NOO}.

\begin{example}
Every group $G$ is a horizontal tribracket under the operation
\[[a,b,c]=ba^{-1}c\]
known as a \textit{Dehn tribracket}. 
\end{example}

\begin{example}
Every $\mathbb{Z}[x^{\pm 1},y^{\pm 1}]$-module is a horizontal tribracket
under the operation
\[[a,b,c]=xb+yc-xya\]
known as an \textit{Alexander tribracket}. 
\end{example}

\begin{example}\label{ex1}
Given a finite set $X=\{1,2,\dots, n\}$ we can specify a tribracket
structure on $X$ with an \textit{operation 3-tensor}, an ordered
list of $n$ $n\times n$ matrices which we may visualize as stacked 
to form a 3-cube. For example,
\[\left[
\left[\begin{array}{rrr}
1 & 3 & 2 \\
2 & 1 & 3 \\
3 & 2 & 1
\end{array}\right],
\left[\begin{array}{rrr}
2 & 1 & 3 \\
3 & 2 & 1 \\
1 & 3 & 2
\end{array}\right],
\left[\begin{array}{rrr}
3 & 2 & 1 \\
1 & 3 & 2\\
2 & 1 & 3
\end{array}\right]
\right]\]
defines a 3-element tribracket where $[a,b,c]$ means the element
in matrix $a$, row $b$ column $c$. For example, $[2,3,1]=1$ and 
$[1,1,2]=3$.
\end{example}

The tribracket axioms are motivated by region coloring for oriented
knots and links according to the rule below:
\[\includegraphics{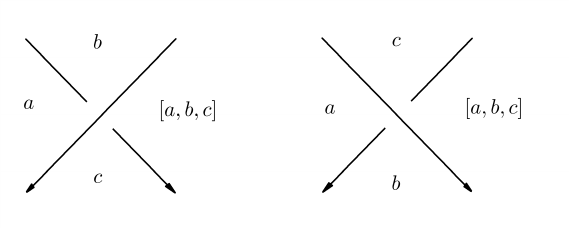}\]

The condition (i) implies that for a given coloring of
a knot or link diagram before a Reidemeister I or II move, there
is a \textit{unique} coloring after the move which agrees with the 
original coloring outside the neighborhood of the move. We illustrate
with the Reidemeister II moves; the Reidemeister I moves are special cases.
\[\raisebox{-0.7in}{\includegraphics{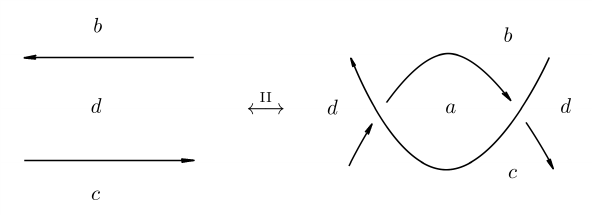}}\qquad 
\forall b,c,d\in X\ \exists!\ a\ s.t.\ [a,b,c]=d \]
\[\raisebox{-0.85in}{\includegraphics{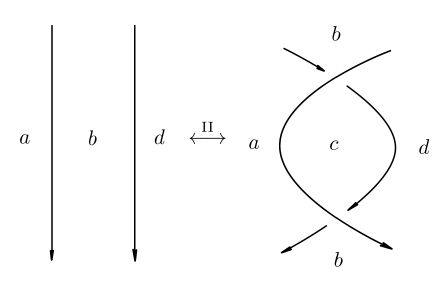}} \qquad
\forall b,c,d\in X\ \exists!\ a\ s.t.\ [a,b,c]=d \]
\[\raisebox{-0.85in}{\includegraphics{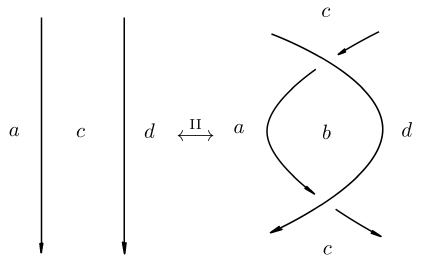}} \qquad
\forall b,c,d\in X\ \exists!\ a\ s.t.\ [a,b,c]=d \]

The condition (ii) ensures the same for Reidemeister III moves:
\[\includegraphics{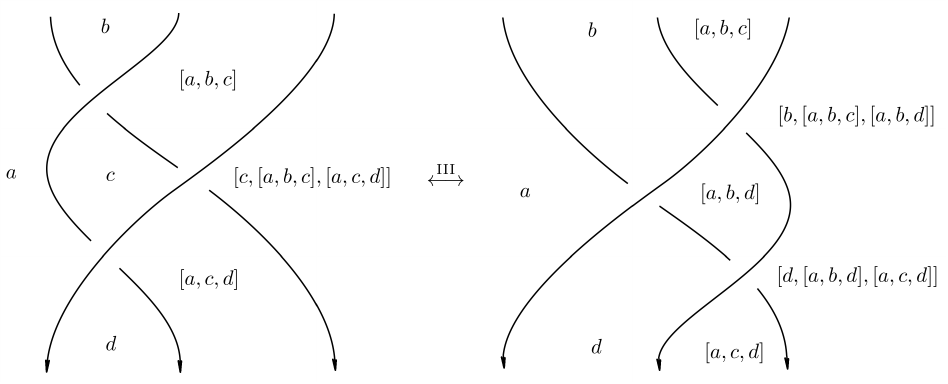}\]

Hence it follows that the number of $X$-colorings of a diagram of a knot or link
$L$ is an integer-valued invariant, known as the 
\textit{tribracket counting invariant}, denoted $\Phi_{X}^{\mathbb{Z}}(L)$.

\begin{example}
Consider colorings of the trefoil knot $K$ and the unknot $U$ by the 
tribracket $X$ in Example \ref{ex1}. 
\[\includegraphics{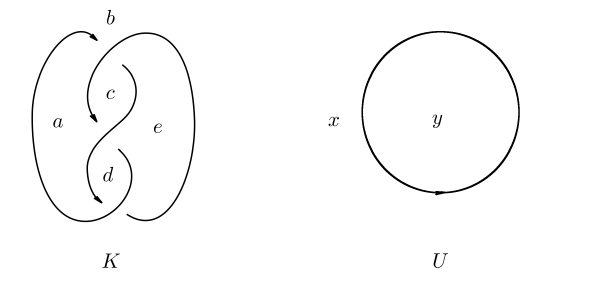}\]
It is easy to see that there are nine 
$X$-colorings of $U$ since we can choose any color $x\in X$ for outside the 
circle and any $y\in X$ for inside; on the other hand, for the trefoil $K$
a check of which possible solutions $(a,b,c,d,e)$ satisfy
the conditions 
\[e=[a,b,c]=[a,c,d]=[a,d,b]\] 
reveals 27 $X$-colorings. 
Then since 
$\Phi_{X}^{\mathbb{Z}}(K)=27\ne9=\Phi_{X}^{\mathbb{Z}}(U)$, the invariant
distinguishes the knots.
\end{example}

We will denote the set of $X$-colorings of an oriented link diagram $L$ by 
$\mathcal{C}_X(L)$; then the counting invariant is the cardinality of
this set, i.e. $\Phi_X^{\mathbb{Z}}(L)=|\mathcal{C}_X(L)|$.

\section{\large\textbf{Tribracket Brackets}}\label{TB}

We will now define a family of quantum enhancements of the tribracket counting
invariant using structures we call \textit{tribracket brackets}.

\begin{definition}
Let $X$ be a tribracket and $R$ a commutative ring with identity. Then a
\textit{tribracket bracket} $\beta$ consists of a pair of maps 
$A,B:X^3\to R^{\times}$ such that
\begin{itemize}
\item[(i)] for all $a,b,c\in X$, the elements 
$-A_{a,b,c}B_{a,b,c}^{-1}-A_{a,b,c}^{-1}B_{a,b,c}$ are all 
equal, with their common value denoted as $\delta$, and
\item[(ii)] for all $a,b,c,d\in X$ we have
\[\begin{array}{rcll}
A_{a,b,c}A_{c,[a,b,c],[a,c,d]}A_{a,c,d} & = &A_{b,[a,b,c],[a,b,d]}A_{a,b,d}A_{d,[a,b,d],[a,c,d]} & (ii.i) \\
A_{a,b,c}B_{c,[a,b,c],[a,c,d]}B_{a,c,d} & = &B_{b,[a,b,c],[a,b,d]}B_{a,b,d}A_{d,[a,b,d],[a,c,d]} & (ii.ii) \\
B_{a,b,c}B_{c,[a,b,c],[a,c,d]}A_{a,c,d} & = &A_{b,[a,b,c],[a,b,d]}B_{a,b,d}B_{d,[a,b,d],[a,c,d]} & (ii.iii) \\
A_{a,b,c}B_{c,[a,b,c],[a,c,d]}A_{a,c,d} & = &A_{b,[a,b,c],[a,b,d]}A_{a,b,d}B_{d,[a,b,d],[a,c,d]} \\
& &+B_{b,[a,b,c],[a,b,d]}A_{a,b,d}A_{d,[a,b,d],[a,c,d]} \\
& &+\delta B_{b,[a,b,c],[a,b,d]}A_{a,b,d}B_{d,[a,b,d],[a,c,d]} \\
& &+B_{b,[a,b,c],[a,b,d]}B_{a,b,d}B_{d,[a,b,d],[a,c,d]} & (ii.iv) \\
A_{a,b,c}A_{c,[a,b,c],[a,c,d]}B_{a,c,d} & & & \\
+B_{a,b,c}A_{c,[a,b,c],[a,c,d]}A_{a,c,d} & & & \\
+\delta B_{a,b,c}A_{c,[a,b,c],[a,c,d]}B_{a,c,d} & & & \\
+B_{a,b,c}B_{c,[a,b,c],[a,c,d]}B_{a,c,d} & = &A_{b,[a,b,c],[a,b,d]}B_{a,b,d}A_{d,[a,b,d],[a,c,d]} & (ii.v).
\end{array}\]
We also define a distinguished element $w=-A_{a,b,b}^2B_{a,b,b}^{-1}$ which is 
required to be the same for all $a,b\in X$.
\end{itemize}
\end{definition}

The tribracket bracket axioms are motivated by the skein relations
\[\includegraphics{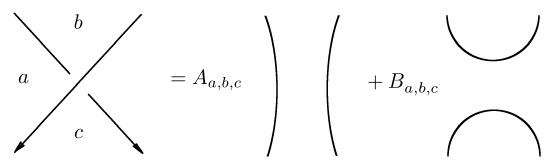}\]
\[\includegraphics{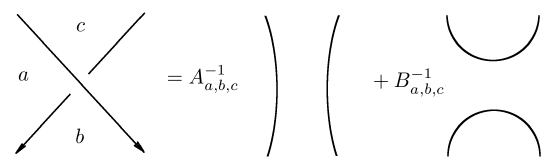}\]
with the elements $\delta$ and $w$ representing the values of a smoothed 
component and a positive kink:
\[\includegraphics{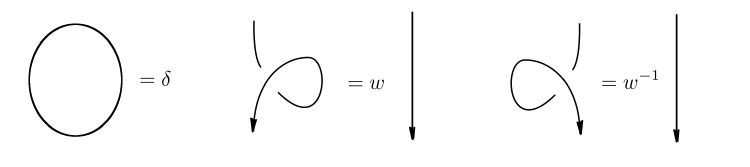}\]

Then the tribracket bracket axioms are the conditions required to make the 
following proposition true:

\begin{proposition}\label{prop1}
Given a tribracket $X$, and $X$-colored link diagram $C$ and tribracket 
bracket $\beta$, we compute an element $\beta(C)$ by first computing the set
of Kauffman states of $C$ and for each state, sum a contribution consisting 
of the product of the smoothing coefficients times $\delta^k$ where $k$
is the number of smoothed components in the state. Finally, we multiply
this sum by the writhe correction factor $w^{n-p}$ where $p$ is the number of
positive crossings and $n$ is the number of negative crossings in $L$. More 
symbolically, that is
\[\beta(C)=w^{n-p}\sum_{\mathrm{states}}\prod 
(\mathrm{smoothing\ coefficients})\delta^k.\]
Then $\beta(C)$is unchanged by $X$-colored Reidemeister moves.
\end{proposition}

This implies our main result:

\begin{corollary}
Given a tribracket $X$ and tribracket bracket $\beta$, for any oriented
link $L$ the multiset
\[\Phi_X^{\beta,M}(L)=\{\beta(C)\ | C\in \mathcal{C}_X(L)\}\]
and the ``polynomial''
\[\Phi_X^{\beta}(L)=\sum_{C\in \mathcal{C}_X(L)}u^{\beta(C)}\]
are invariants of oriented links.
\end{corollary}

\begin{proof} (of Proposition \ref{prop1})
For each of the moves in our generating set of oriented Reidemeister moves,
we compare the contributions to $\beta(C)$ on both sides of the move and see
that they agree.

First, applying the skein relation at either positive $X$-colored 
Reidemeister I move 
\[\includegraphics{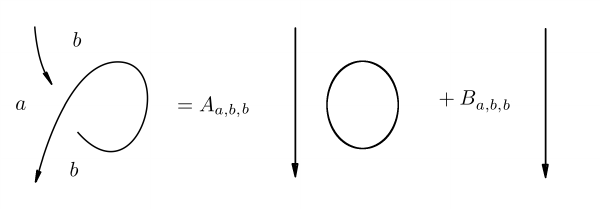}\]
we obtain the requirement that
\begin{eqnarray*}
w & =& A_{a,b,b}\delta +B_{a,b,b} \\
 & =& A_{a,b,b}(-A_{a,b,b}^{-1}B_{a,b,b}-A_{a,b,b}B_{a,b,b}^{-1}) +B_{a,b,b} \\
 & =& -B_{a,b,b}-A_{a,b,b}^2B_{a,b,b}^{-1} +B_{a,b,b} \\
 & =& -A_{a,b,b}^2B_{a,b,b}^{-1}
\end{eqnarray*}
and a similar computation at either negative Reidemeister move yields
$w^{-1}=-A_{a,b,b}^{-2}B_{a,b,b}$.

Next, comparing coefficients of the smoothed tangles
on both sides of each Reidemeister II move 
\[\includegraphics{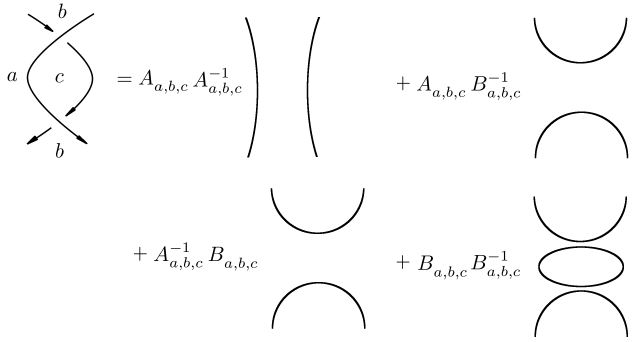}\]
yields the requirement that $\delta=-A_{a,b,b}B_{a,b,b}^{-1}-A_{a,b,b}^{-1}B_{a,b,b}$.
Here we depict one of the four oriented Reidemeister II moves; the
others yield the same result.

Finally, in the Reidemeister III in our generating set in Section \ref{B}, 
on the left side we have
\[\includegraphics{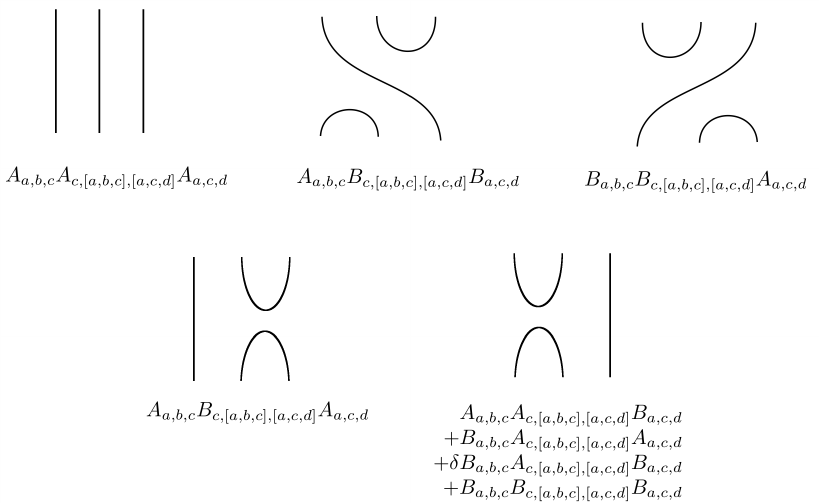}\]
and on the right side, we have
\[\includegraphics{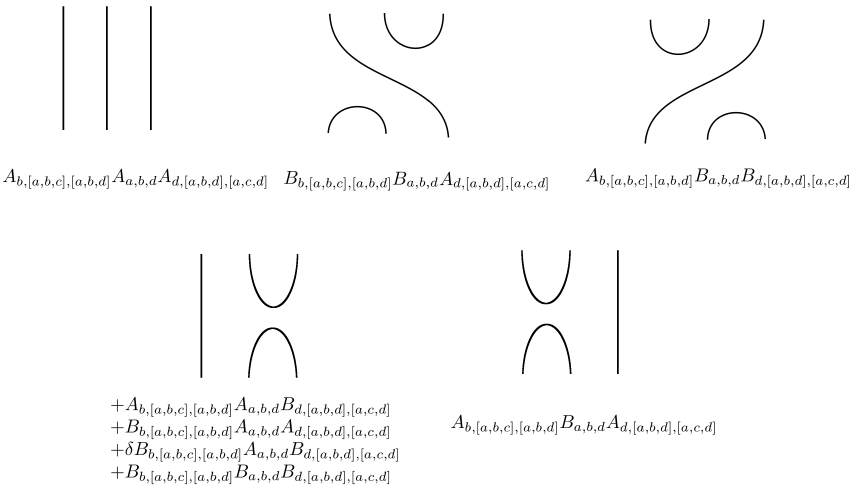}\]
Comparing coefficients then yields the equations (ii.i)-(ii.v).
\end{proof}

For a finite tribracket $X=\{1,2,\dots, n\}$ we can specify a tribracket 
bracket $\beta$ with a pair of 3-tensors $A$ and $B$ whose entries are 
the $A_{x,y,z}$ and $B_{x,y,z}$ coefficients:
\[A=\left[
\left[\begin{array}{rrrr} 
A_{1,1,1} & A_{1,1,2} & \dots & A_{1,1,n} \\
A_{1,2,1} & A_{1,2,2} & \dots & A_{1,2,n} \\
\vdots & \vdots & \ddots & \vdots \\
A_{1,n,1} & A_{1,n,2} & \dots & A_{1,n,n} \\
\end{array}\right],
\dots,
\left[\begin{array}{rrrr} 
A_{n,1,1} & A_{n,1,2} & \dots & A_{n,1,n} \\
A_{n,2,1} & A_{n,2,2} & \dots & A_{n,2,n} \\
\vdots & \vdots & \ddots & \vdots \\
A_{n,n,1} & A_{n,n,2} & \dots & A_{n,n,n} \\
\end{array}\right]\right]\]
and
\[B=\left[
\left[\begin{array}{rrrr} 
B_{1,1,1} & B_{1,1,2} & \dots & B_{1,1,n} \\
B_{1,2,1} & B_{1,2,2} & \dots & B_{1,2,n} \\
\vdots & \vdots & \ddots & \vdots \\
B_{1,n,1} & B_{1,n,2} & \dots & B_{1,n,n} \\
\end{array}\right],
\dots,
\left[\begin{array}{rrrr} 
B_{n,1,1} & B_{n,1,2} & \dots & B_{n,1,n} \\
B_{n,2,1} & B_{n,2,2} & \dots & B_{n,2,n} \\
\vdots & \vdots & \ddots & \vdots \\
B_{n,n,1} & B_{n,n,2} & \dots & B_{n,n,n} \\
\end{array}\right]\right].\]

\begin{example}
If $X=\{1\}$ then setting $A_{1,1,1}=A$ and $B_{1,1,1}=A^{-1}$ gives us 
$\Phi_X^{\beta}(L)$ equal to a normalization of the
 Kauffman bracket/Jones polynomial. 
\end{example}

\begin{example}
If $A_{x,y,z}=B_{x,y,z}$ for all $x,y,z\in X$ then $\beta$ is a tribracket 
2-cocycle and $\Phi_X^{\beta}(L)$ equals the tribracket 2-cocycle invariant
times the Kauffman bracket polynomial evaluated at $A=-1$.
\end{example}

\begin{example}\label{ex2}
Let $X$ be the tribracket specified by the 3-tensor
\[\left[
\left[\begin{array}{rr} 2 & 1 \\1 & 2\end{array}\right],\ 
\left[\begin{array}{rr} 1 & 2 \\2 & 1\end{array}\right]\right]\]
Then our \texttt{python} computations reveal tribracket bracket structures
with $\mathbb{Z}_7$ coefficients including
\[A=
\left[
\left[\begin{array}{rr} 1 & 3 \\2 & 1\end{array}\right],\ 
\left[\begin{array}{rr} 1 & 2 \\3 & 1\end{array}\right]\right]
\ \mathrm{and} \ B=
\left[
\left[\begin{array}{rr} 5 & 1 \\3 & 5\end{array}\right],\ 
\left[\begin{array}{rr} 5 & 3 \\1 & 5\end{array}\right]\right].\]
For instance here we have $A_{1,2,2}=1$ and $B_{2,1,2}=3$.
This tribracket bracket has 
\[\delta=-A_{111}^{-1}B_{111}-A_{111}B_{111}^{-1}=-(1)(5)-1(3)=-8=6\]
and
\[w=-A_{1,1,1}^2B_{111}^{-1}=-(1)^2(3)=-3=4\]
and satisfies sixteen skein relations including
\[\includegraphics{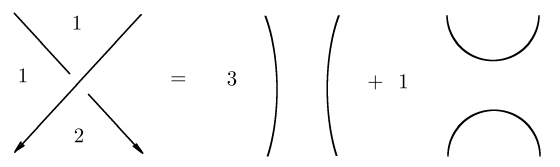}\]
and
\[\includegraphics{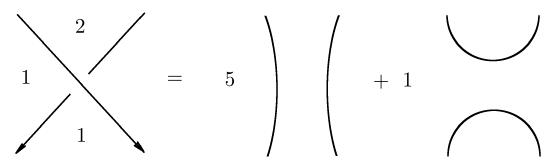}.\]
\end{example}

\begin{example}
The tribracket bracket polynomial invariant $\Phi_X^{\beta}$ associated to the 
tribracket bracket in example \ref{ex2}
distinguishes the Hopf link from the $(4,2)$-torus link with invariant values 
$4u^6+4u$ and $8u$ respectively, demonstrating that the enhancement is proper. 
Let us illustrate the computation of the invariant with the Hopf link. First, 
we find the set of all $X$-colorings of the link:
\[\scalebox{0.85}{\includegraphics{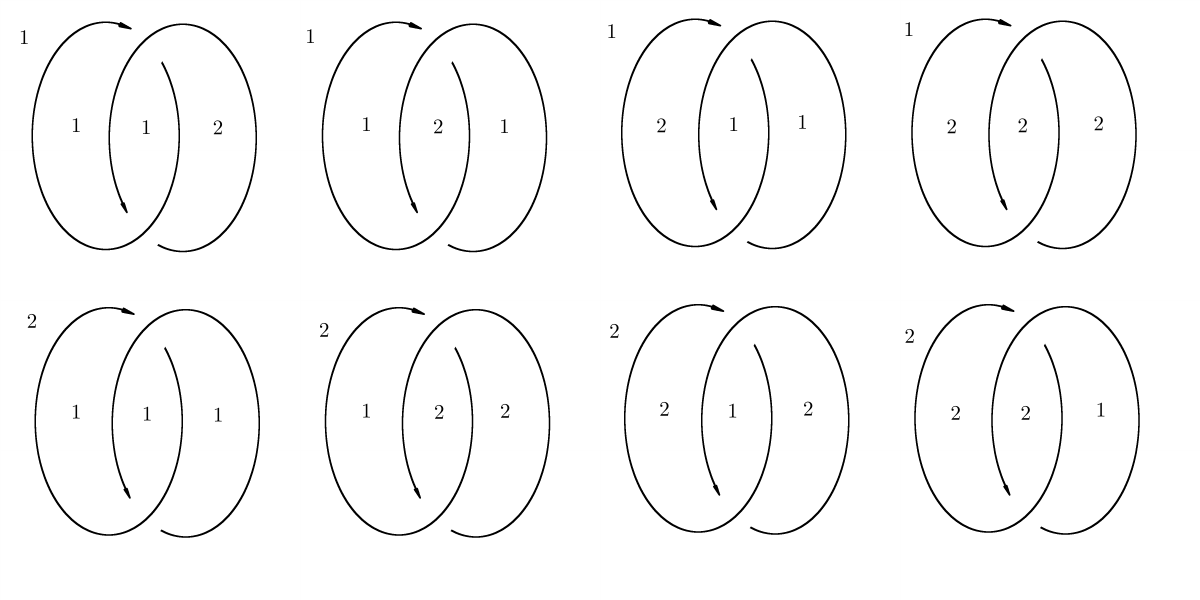}}.\]

Next, for each coloring $C$, we compute the state-sum value $\beta(C)$. We
illustrate the process with a choice of coloring: 
\[\scalebox{0.85}{\includegraphics{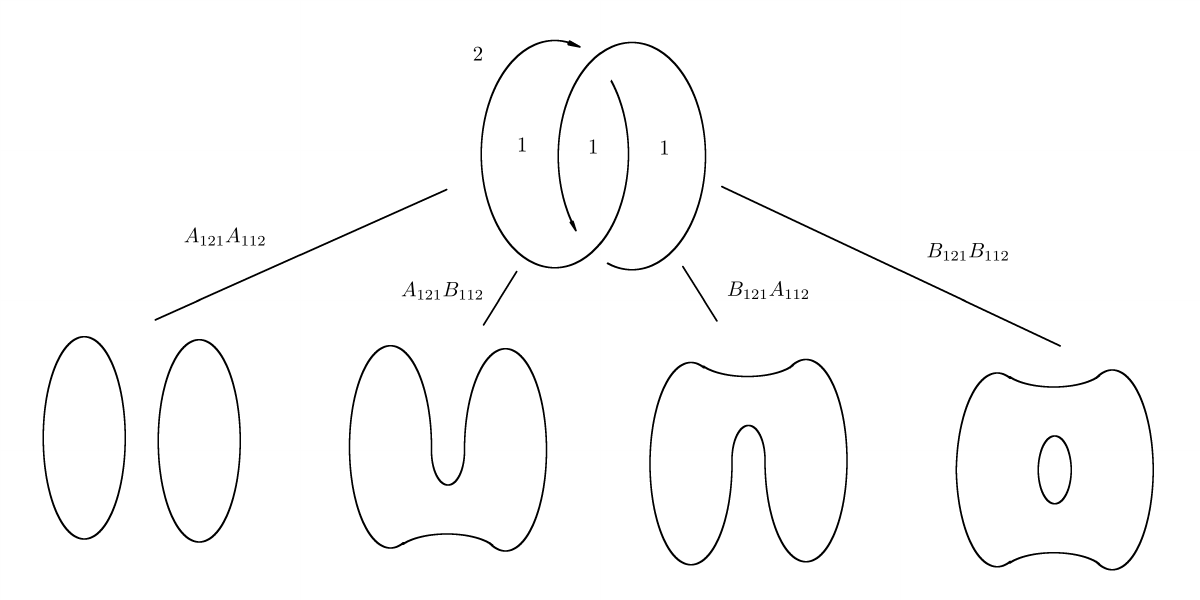}}\]
Thus for this coloring we obtain state-sum value
\begin{eqnarray*}
\beta(C) 
& = & w^{-2}(A_{121}A_{112}\delta^2+A_{121}B_{112}\delta+ B_{121}A_{112}\delta+ B_{121}B_{112}\delta^2)\\
& = & 4^{-2}((2)(3)(6^2)+(2)(1)(6)+ (3)(3)(6)+ (3)(1)(6)^2)\\
& = & 4(6+5+ 5+ 3)=4(5)=6
\end{eqnarray*}
and this coloring contributes $u^6$ to $\Phi_X^{\beta}(\mathrm{Hopf\ link})$.
Repeating for the other colorings and for the $(4,2)$-torus link, we obtain 
the stated result.
\end{example}

\begin{example}
Using \texttt{python} code, we computed the tribracket bracket polynomial 
for all of the prime knots 
with up to 8 crossings and prime links with up to 7 crossings as found in
\cite{KA} (with a choice of orientation for each component) with respect
to the tribracket $X$ from example \ref{ex2} and the $\mathbb{Z}_5$ 
tribracket bracket $\beta$ given by
\[A=
\left[
\left[\begin{array}{rr} 1 & 4 \\2 & 1\end{array}\right],\ 
\left[\begin{array}{rr} 1 & 1 \\3 & 1\end{array}\right]\right]
\ \mathrm{and} \ B=
\left[
\left[\begin{array}{rr} 4 & 1 \\3 & 4\end{array}\right],\ 
\left[\begin{array}{rr} 4 & 4 \\2 & 4\end{array}\right]\right].\]
The results are collected in the table. We note that 
this invariant distinguishes the square knot $SK$ from the granny knot 
$GK$ with $\Phi_X^{\beta_1}(SK)=u^4+3u^2\ne 4u^2=\Phi_X^{\beta_1}(GK)$. 
\[
\begin{array}{l|l}
\Phi_X^{\beta_1}(L) & L \\\hline
4u^2 & 0_1, 3_1, 5_1, 5_2, 6_1, 6_2, 7_1, 7_3, 7_4,\\
     & 8_3, 8_4, 8_9, 8_{12}, 8_{15}, 8_{17}, 8_{18}, 8_{19} \\
u^3+3u^2 & 4_1, 8_1, 8_2, 8_5, 8_6, 8_{11}, 8_{14} \\
u^4+3u^2 & 6_3, 7_2, 7_5, 7_7, 8_7, 8_8, 8_{10}, 8_{13} \\
u^4+u^3+2u^2 & 8_{20}\\
2u^4+2u^2 & 8_{21} \\
3u^2+u & 7_6, 8_{16} \\
3u^4+2u^2+3u & L7n1\\
4u^4+4u & L4a1, L6a1, L7a2 \\
4u^4+4u^2 & L2a1 \\
4u^4+4u^3 & L6a2,L6a3 \\
4u^4+3u^3+u^2 & L7a3, L7a4 \\
4u^4+u^3+3u^2 & L7a6\\
5u^4+2u^2+u & L7a1, L7a5, L7n2 \\ 
6u^4+2u^3 & L5a1 \\
3u^4+5u^3+3u^2+5u & L7a7\\
4u^3+12u^2 & L6a5, L6n1 \\
12u^3+4u & L6a4 \\
\end{array}\]
\end{example}

\begin{remark}
We have defined tribracket brackets analogously to biquandle brackets as in
\cite{NOR}, using the state-sum formulation in which we do all smoothings
simultaneously to obtain a completed set of states. With classical skein 
invariants one often wants to compute the invariant recursively by doing
Reidemeister moves on partially smoothed diagrams. In \cite{NO} a method
for such computations is defined for biquandle brackets using \textit{trace
diagrams}, decorated trivalent graphs which can be given biquandle colorings
for partially smoothed states. We can make an analogous construction with 
tribracket brackets by defining tribracket colorings of trace diagrams in the
following way:
\[\includegraphics{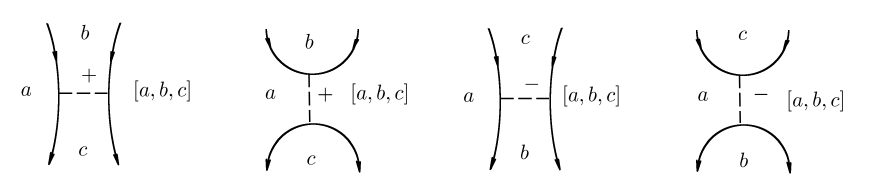}\]
We then have skein relations
\[\includegraphics{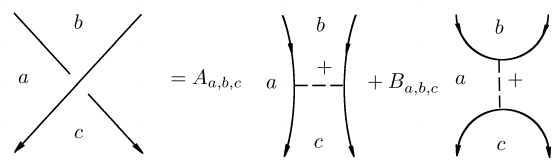}\]
\[\includegraphics{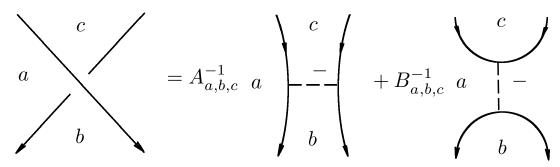}\]
with the value of a crossingless diagram given by $\delta^kw^{n-p}$ where
$n$ is the number of negative signed traces, $p$ is the number of positive
signed traces, and $k$ is the number of components of the bivalent graph 
obtained by deleting all traces.

As observed in \cite{NO} for the biquandle bracket case, for any 
biquandle bracket there is some subset of possible \textit{trace moves},
i.e. moving a strand over, under or through a trace, 
which are allowed and some which are not depending on whether certain
conditions are satisfied by the bracket; an analogous result holds in 
the tribracket bracket case.
\end{remark}

\section{\large\textbf{Questions}}\label{Q}

We conclude with a few questions for future research. 

In the case of biquandle brackets over a ring $R$, some brackets represent 
cocycles in the second cohomology of the coloring biquandle with coefficients 
in $R$, and cohomologous cocycles define the same invariant. To what extent is 
the same thing true in the tribracket bracket case? Indeed, what is the 
relationship between tribracket brackets and biquandle brackets?

As with biquandle brackets, what other quantum enhancements of the tribracket
counting invariant are there beyond those which can be represented as 
tribracket brackets?

For ease of computation, we have considered only examples of tribracket 
brackets over finite rings;
what are some good examples of tribracket brackets over infinite rings
such as $\mathbb{Z}$ or $\mathbb{Z}[t]$?

What is the geometric meaning of a tribracket bracket? How can we categorify
tribracket brackets?

\bigskip

\bibliography{la-sn-pr}{}
\bibliographystyle{abbrv}

\noindent
\textsc{Department of Mathematical Sciences \\
Claremont McKenna College \\
850 Columbia Ave. \\
Claremont, CA 91711} 

\end{document}